\documentclass[aos,preprint]{imsart}

\usepackage[english]{babel}
\usepackage[utf8]{inputenc}
\RequirePackage[OT1]{fontenc}
\RequirePackage{amsthm,amsmath}
\RequirePackage[numbers]{natbib}
\RequirePackage[colorlinks,citecolor=blue,urlcolor=blue]{hyperref}

\usepackage{amssymb}
\usepackage{mathtools}
\usepackage{graphicx,enumerate,subfigure}
\usepackage[colorinlistoftodos]{todonotes}
\usepackage{amsfonts}
\usepackage{tikz,url}
\arxiv{arXiv:0000.0000}

\startlocaldefs
\numberwithin{equation}{section}
\theoremstyle{plain}
\newtheorem{thm}{Theorem}[section]
\newtheorem{prop}{Proposition}[section]
\newtheorem {lema}{Lemma}[section]
\newtheorem {exa}{Example}[section]
\newtheorem {rem}{Remark}[section]
\endlocaldefs

\newcommand{\dd}{\,\mathrm{d}}

\begin{document}

\begin{frontmatter}
\title{On particle-size distribution of convex similar bodies in $\mathbb{R}^3$} 
\runtitle{On particle-size distribution of convex similar bodies}

\begin{aug}
\author{\fnms{Jozef} \snm{Kise\v l\'ak}\thanksref{t1,m1,m2}\ead[label=e1]{jozef.kiselak@upjs.sk}},

\author{\fnms{Gabriela}\snm{Bal\'uchov\'a}\thanksref{m1}\ead[label=e2]{gabriela.baluchova@student.upjs.sk}}

\thankstext{t1}{Supported by the Slovak Research and Development
Agency under the contract No. APVV-16-0337.}
\runauthor{J. Kise\v l\'ak et al.}

\affiliation{P.J. \v Saf\'arik University\thanksmark{m1}, Johannes Kepler University\thanksmark{m2}}

\address{
Institute of Mathematics , P.J. \v Saf\'arik University in Ko\v sice, Slovakia \\
LIT and IFAS, Johannes Kepler University Linz, Austria\\
\printead{e1}\\
 \phantom{E-mail:\ }\printead*{e2}}

\end{aug}

\begin{abstract}\\
We give an explicit form of particle-size distributions of convex similar bodies
for random planes and random lines, which naturally generalize famous Wicksell's corpuscle problem. The results are achieved by applying the Method of Model Solutions for solving well-known Santal\'o's integral equations. We also give a partial result related to the question of existence and uniqueness of these solutions. We finally illustrate our approach on several examples.
\end{abstract}



\end{frontmatter}

\section{Introduction}

Stereology is originally concerned with the determination of three-dimensional structure from two-dimensional or one-dimensional observational data. It provides practical techniques for extracting quantitative information about this structure and is based on fundamental principles of geometry and statistics. It is a completely different approach from computed tomography, for which a complete set of all cross sections is needed. 
E.g. \cite{SfS} sets out the principles of stereology from a statistical viewpoint, focusing on both basic theory and practical implications.  It is an important and efficient tool in many applications of geology, metallurgy, petrology but also cell biology, petrography, materials science,  histology or neuroanatomy.

In this article the general goal is the reconstruction of particles from cross sections in $\mathbb{R}^3$. We focus only on Santal\'o's formulation, i.e. on the estimation problems concerned with ascertaining the size distribution of similarly shaped convex particles, capable of complete size specification by one size parameter and randomly distributed in a convex opaque field, see \cite{Santalo} or \cite{topics}. It also involves a famous original Wicksell’s corpuscle problem concerns the determination of the distribution of spherical particles from planar sections, see \cite{WICKSELL}. The solution was derived already in \cite{WICKSELL}, but proven much later in \cite{Ken}. This paper is organized as follows. We first present preliminaries in Section \ref{setup}. In Section \ref{PSD}, we give an overview of the problem and introduce related integral equations that we want to solve. Sections \ref{rovina} and \ref{priamka} obtain
most important result, an exact solutions for similar convex bodies for random planes and lines. The question of solvability is partially solved in Section \ref{exist}. We illustrate our approach on several explicitly solved examples, see Sections \ref{RPex} and \ref{RLex}. We put necessary technicalities in appendix.

\section{Setup and preliminaries}\label{setup}

Let $\Omega$ be non-empty set. A complex random variable $Z$ on the probability space $ (\Omega ,{\mathcal {F}},\mathbb{P})$ is a map ${\displaystyle Z\colon \Omega \rightarrow \mathbb {C} }$ such that $Z=\Re(Z)+\mathrm{i}\,\Im(Z)$, where ${\displaystyle \Re {(Z)}}$ a  ${\displaystyle \Im {(Z)}}$ are real random variables on $(\Omega ,{\mathcal {F}},\mathbb{P})$
(it can always be considered as pair of real random variables: its real and imaginary part).
The $w$-weighted $L^p$ space for measurable set $A\subseteq \mathbb{R}$ is defined as
$$L^p_w(A):=\left\{f\in  M(A)\,: \,||f||_{L^p_w(A)} <\infty \right\}, ~1\leq p<\infty,$$
where $||f||_{L^p_w(A)}:=\left(\int\limits_A w(x)|f(x)|^p\,\dd x\right)^\frac{1}{p}$, $w$ is a non-negative measurable function (weight) on $A$
and $M(A)$ the set of all Lebesgue measurable functions. 
For simplicity, if $p=1$, $w\equiv 1$ we do not write superscript or subscript. We also denote for fixed $\mu\in\mathbb{R}$ the space
$$L^p_{\{\mu\}}(\mathbb{R}_+):=\{f: \mathbb{R}_0\to\mathbb{C} ~: ||f||_{L^p_{v}(\mathbb{R}_+)}<\infty\},$$
where $v(x)=x^{\mu\,p-1}$, whereas it it known that additionally an isometry $||f||_{L^p_{\{\mu\}}(\mathbb{R}_+)}:=||f||_{L^p_{v}(\mathbb{R}_+)}=||v^{1/p}\,f||_{L^p(\mathbb{R})}$ holds. Similarly we set $L^p_{(a,b)}(\mathbb{R}_+)=\displaystyle\bigcap_{\mu\in(a,b)}L^p_{\{\mu\}}(\mathbb{R}_+), ~-\infty\leq a<b\leq \infty$.

Now, recall that $k$-th moment of real random variable $X$
is defined  as
$${\displaystyle \mathbb{E}[X^k]=\int\limits\limits\limits\limits\limits\limits\limits\limits\limits\limits\limits _{\Omega }X(\omega )^k\,\mathrm{d}\mathbb{P} (\omega )}=\int\limits\limits\limits\limits\limits\limits\limits\limits\limits\limits\limits _{\mathbb{R} }x^k\,\mathrm{d}\operatorname {F} (x),$$
if the integral exists.  Let us consider only random variables whose probability distributions are absolutely continuous, i.e. for  distribution function we have $\mathbb{P}(X\leq x)=F(x)=\int\limits\limits\limits\limits\limits\limits\limits\limits\limits\limits\limits _{-\infty }^{x}f(t)\,dt$ for $\forall x\in \mathbb{R}.$
Here $f$ is a probability density function\footnote{Note that $ { f (x) \,\mathrm{d}x} $ can be viewed as the probability that the random variable $ X $ will fall within the infinitesimal interval $ [x, x + \mathrm{d}x].$}, i.e. non-negative function  from $L(\mathbb{R})$ such that $||f||_{L(\mathbb{R})}=1$. It is good to realize that the Mellin transform defined by \eqref{MT} in appendix \ref{MT} of $f$ is $s-1$-th moment of $X$. In that case one has
${\displaystyle \mathbb {E} [X^k]=\int\limits\limits\limits\limits\limits\limits\limits\limits\limits\limits\limits _{\mathbb {R} }x^kf(x)\,\mathrm{d}x}$ and necessarily $f\in L_w(\mathbb{R}), ~w(x)=x^k$. More generally for (Borel) measurable function $r:\mathbb{R}\to\mathbb{C}$ it holds, that $r(X)=\Re(r(X))+\mathrm{i}\,\Im(r(X))$ is a complex random variable and   
$$\mathbb{E}[r(X)]=\mathbb{E}[\Re(r(X))]+\mathrm{i}\,\mathbb{E}[\Im(r(X))],$$
with
$$\mathbb{E}[r(X)]=\int\limits\limits\limits\limits\limits\limits\limits\limits\limits\limits\limits _{\mathbb{R} }r(x)\,\mathrm{d}\operatorname {F} (x)=\int\limits\limits\limits\limits\limits\limits\limits\limits\limits\limits\limits _{\mathbb{R} }\Re(r(x))\,\mathrm{d}\operatorname {F} (x)+\mathrm{i}\,\int\limits\limits\limits\limits\limits\limits\limits\limits\limits\limits\limits _{\mathbb{R} }\Im(r(x))\,\mathrm{d}\operatorname {F} (x)$$
(in the case of integrability). So for absolutely continuous complex random variable $Z$ one has 
$$\mathbb{E}[r(X)]=\int\limits\limits\limits\limits\limits\limits\limits\limits\limits\limits\limits _{\mathbb{R} }\Re(r(x))\,f(x)\,\mathrm{d}x+\mathrm{i}\,\int\limits\limits\limits\limits\limits\limits\limits\limits\limits\limits\limits _{\mathbb{R} }\Im(r(x))\,f(x)\,\mathrm{d}x.$$
We are dealing only with random variables with specific finite support, i.e.
$$\mathcal{S}_c:=\left\{f\in L(\mathbb{R})\,:f\geq 0, ||f||_{L(\mathbb{R})}= 1, ~\operatorname {supp}(f)=I_c \right\},$$
where $\operatorname {supp} (f):=\{x\in \mathbb{R}\,|\,f(x)\neq 0\}$ and $I_c=[0,c], ~I_c=(0,c], ~I_c=(0,c)$, or $I_c=[0,c)$ for some $0<c<\infty$. Obviously $\mathcal{S}_c\subset L(\mathbb{R})\subset L_{\mathrm{loc}}(\mathbb{R})$. We start with the following.

\begin{lema}\label{lem1}
Let $f\in\mathcal{S}_c$ be a density of random variable $X$. Then the moments $\mathbb{E}[X^\alpha], ~\alpha\geq 0$ does exist, are positive and
 \begin{enumerate}[i)]
\item $\mathbb{E}[X^\alpha]\leq c^\alpha$ for $\alpha\geq 0$;
     \item $\mathbb{E}[X]^\alpha\leq \mathbb{E}[X^\alpha]$ for $\alpha\geq 1$.
 \end{enumerate}
 Moreover, for $\alpha<0$ is moment a positive or $+\infty$. Also, $f=\mathcal{O}(x^{-\gamma_f}), ~x\to0^+, ~\gamma_f<1$ and $f=\mathcal{O}(x^{-\beta_f}), ~x\to\infty, ~\forall \beta_f>0$. 
\end{lema}

\begin{proof}
Positiveness results from the fact that for $X>0$ (a.s.) one has $\mathbb{E}[X]>0.$ 
Furthermore, for $\alpha$ we obtain $\int\limits\limits\limits\limits\limits\limits\limits\limits\limits\limits\limits_0^c x^\alpha f(x)\,\mathrm{d}x\leq c^\alpha\int\limits\limits\limits\limits\limits\limits\limits\limits\limits\limits\limits_0^c  f(x)\,\mathrm{d}x=c^\alpha.$ The estimate from below follows from Jensen's inequality. Asymptotic property for $x\to\infty$ follows from boundedness of the support and for $x\to0^+$ from the fact that $f\in L(0,c).$
\end{proof}

\section{Particle-size distribution of a body based on its cut}\label{PSD}

Consider a convex body (compact convex set with non-empty interior) $\mathbf{Q}$ v $\mathbb{R}^3$ containing a certain number of randomly placed, non-overlapping particles. Suppose that all of them are similar\footnote{By similarity on a metric space $(X,d)$ we assume a bijection $f: X\to X$ such that $\forall x,y\in X ~~d(f(x),f(y)) = \lambda\, d(x,y)$ for fixed $\lambda>0$. We call two sets similar if one is an image of the other according to similarity $f$.} to some convex body $\mathbf{K}$ and their coefficient of similarity is $\lambda>0$. We denote a particle similar to $\mathbf{K}$ as $\mathbf{K}_{\lambda}$, so that $\mathbf{K}_1=\mathbf{K}$. $V_\lambda$ denotes its volume, $F_\lambda$ surface area and $M_\lambda$ the integral of the mean curvature\footnote{I.e. the average of the principal curvatures or equivalently defined using a divergence of the unit normal.} of $\partial \mathbf{K}_\lambda$ (its boundary). We omit the subscript when $\lambda=1$. Moreover, let $H\left(\lambda\right)\mathrm{d}\lambda$ be a number of particles per unit of volume in $\mathbf{Q}$, whose coefficient of similarity is within the range $\left(\lambda, \lambda + \dd\lambda\right).$

We intersect body $\mathbf{Q}$ with a random plane $\mathbf{E}$ and a random line $\mathbf{G}$. 
Then the particles contained in $\mathbf{Q}$, intersected by  $\mathbf{E}$ ($\mathbf{G}$), are the sections, convex domains in $\mathbf{E}\cap\mathbf{Q}$
(intervals on $\mathbf{G}\cap\mathbf{Q}$). Let
$h\left(\sigma\right)\mathrm{d}\sigma$ and 
$h\left(l\right)\mathrm{d}l$ be the number of these sections of area and length in the range $\left(\sigma, \sigma + \dd\sigma\right)$ per unit area and $(l,l+\dd l)$ per unit length, respectively.
These values are random variables with densities $h(\sigma)$ and $h(l)$. However the area of $\mathbf{E}\cap\mathbf{Q}$ and 
length of $\mathbf{G}\cap\mathbf{Q}$ are also random variables with densities\footnote{$\phi\left(\sigma\right)\dd\sigma$ is the probability that the area of $\mathbf{E}\cap\mathbf{Q}$ lies between $\sigma $ and $ \sigma+\dd\sigma$. $\phi\left(l\right)\dd l$ is the probability that the length of $\mathbf{G}\cap\mathbf{Q}$ lies between $l$ and $ l+\dd l$.}  $\phi\left(\sigma\right)$ and  $\phi\left(l\right)$. Set $\phi\left(\sigma, \lambda\right)$ as a density of the area of $\mathbf{E \cap K_{\lambda}}$ and $\phi\left(l, \lambda\right)$ as a density of the length of $\mathbf{G \cap K_{\lambda}}$,  naturally $\phi\left(\sigma, 1\right) = \phi\left(\sigma\right)$. Moreover we have
\begin{equation}\label{FI}
 \phi\left(\sigma, \lambda\right) = \frac{1}{\lambda^{2}}  \phi\left(\frac{\sigma}{ \lambda^{2}} \right),
\end{equation}
\begin{equation}\label{FI2}
 \phi\left(l, \lambda\right) = \frac{1}{\lambda}  \phi\left(\frac{l}{\lambda} \right),
\end{equation}
see  \cite{Santalo} or \cite{topics}. To find a $\phi$ is not in general a simple task. However, we assume that it is given and our problem is to find $H\left(\lambda\right)$ from $h\left(\sigma\right)$ and $h\left(l\right)$ respectively. Notice that we have finite mean values of $\sigma $ and $l$ respectively as they equal the following  
$$\int\limits\limits\limits\limits\limits\limits\limits\limits\limits\limits\limits_{\mathbf{E} \cap \mathbf{K} \neq 0}\sigma(\mathbf{E \cap K})\, \mathrm{d}\mathbf{E} \Big{/} \int\limits\limits\limits\limits\limits\limits\limits\limits\limits\limits\limits_{\mathbf{E} \cap \mathbf{K} \neq 0} \mathrm{d} \mathbf{E}=\frac{2\pi\textit{V}}{\textit{M}},$$

$$
    \int\limits\limits\limits\limits\limits\limits\limits\limits\limits\limits\limits_{\mathbf{G} \cap \mathbf{K} \neq 0} l \left( \mathbf{G} \cap \mathbf{K} \right) \mathrm{d}\mathbf{G} \Big{/} \int\limits\limits\limits\limits\limits\limits\limits\limits\limits\limits\limits_{\mathbf{G} \cap \mathbf{K} \neq 0} \mathrm{d} \mathbf{G} = 2\pi V \Big{/} \left( \frac{\pi}{2} F \right) = \frac{4V}{F}.
$$

\subsection{Random plane intersection}

For spherical case, see illustrative Figure \ref{rovF}. Denote as $\sigma_{m}$ the maximal admissible value of  $\sigma$, i.e. $\displaystyle\sigma_m=\max\sigma(\mathbf{E \cap K})$. We thus have that $\phi\in \mathcal{S}_{\sigma_m}$, and therefore
\begin{equation}\nonumber
\int\limits\limits\limits\limits\limits\limits\limits\limits\limits\limits\limits\limits_{0}^{\sigma_{m}}\phi\left(\sigma\right)\dd\sigma= 1, \qquad \int\limits_{0}^{\sigma_{m}}\sigma\,\phi\left(\sigma\right)\dd\sigma=  \frac{2\pi\textit{V}}{\textit{M}}.
\end{equation}

\begin{exa}[Spherical case]\label{SCE}
In spherical case, see e.g. \cite{Santalo},  one has $\sigma_{m} = \pi$ ($M = 4\pi$) and $\lambda$ is the radius of particle $\mathbf{K_{\lambda}}$. Then
\begin{equation}
   \phi(\sigma) = \frac{1}{2\sqrt{\pi}\left( \pi - \sigma \right)^\frac{1}{2}}.
\end{equation}
or
\begin{equation}
   \phi(\sigma, \lambda) = \frac{\lambda}{2\sqrt{\pi}\left( \pi\lambda^2 - \sigma \right)^\frac{1}{2}}. 
\end{equation}
\end{exa}

In \cite{Santalo} is derived the integral relationship between $H$ and $h$. We have for planes equation, which is of interest in this work:

\begin{equation}\label{RNR}\tag{RP}
\boldmath { \int\limits\limits\limits\limits\limits\limits\limits\limits\limits\limits\limits\limits_{\left(\frac{\sigma}{\sigma_{m}}\right)^\frac{1}{2}}^{\infty} \phi\left(\frac{\sigma}{\lambda^2}\right)\frac{H\left(\lambda\right)}{\lambda}\dd\lambda = \frac{1}{\alpha} \; h\left(\sigma\right)},
\end{equation}
where  $\alpha$ is a constant (depending on e.g. $M$). Equation \eqref{RNR} is a Volterra integral equation, where function $\phi\left(\sigma\right)$ depends on the shape of the body $\mathbf{K}$ (and thus on the shape of all particles). Using the cross-sectional measurements of the body $ \mathbf {Q} $, it is possible to estimate the function $ h $, and then to derive $H$ from \eqref{RNR}. Notice if one finds $H$ then also has a distribution in $\sigma$ given by $H_1(\sigma)=H\left(\sqrt{\sigma/\sigma_m}\right)/\sqrt{4\sigma_m\sigma}$.

\begin{figure}[h!]
    \centering
    \includegraphics[width=1\textwidth]{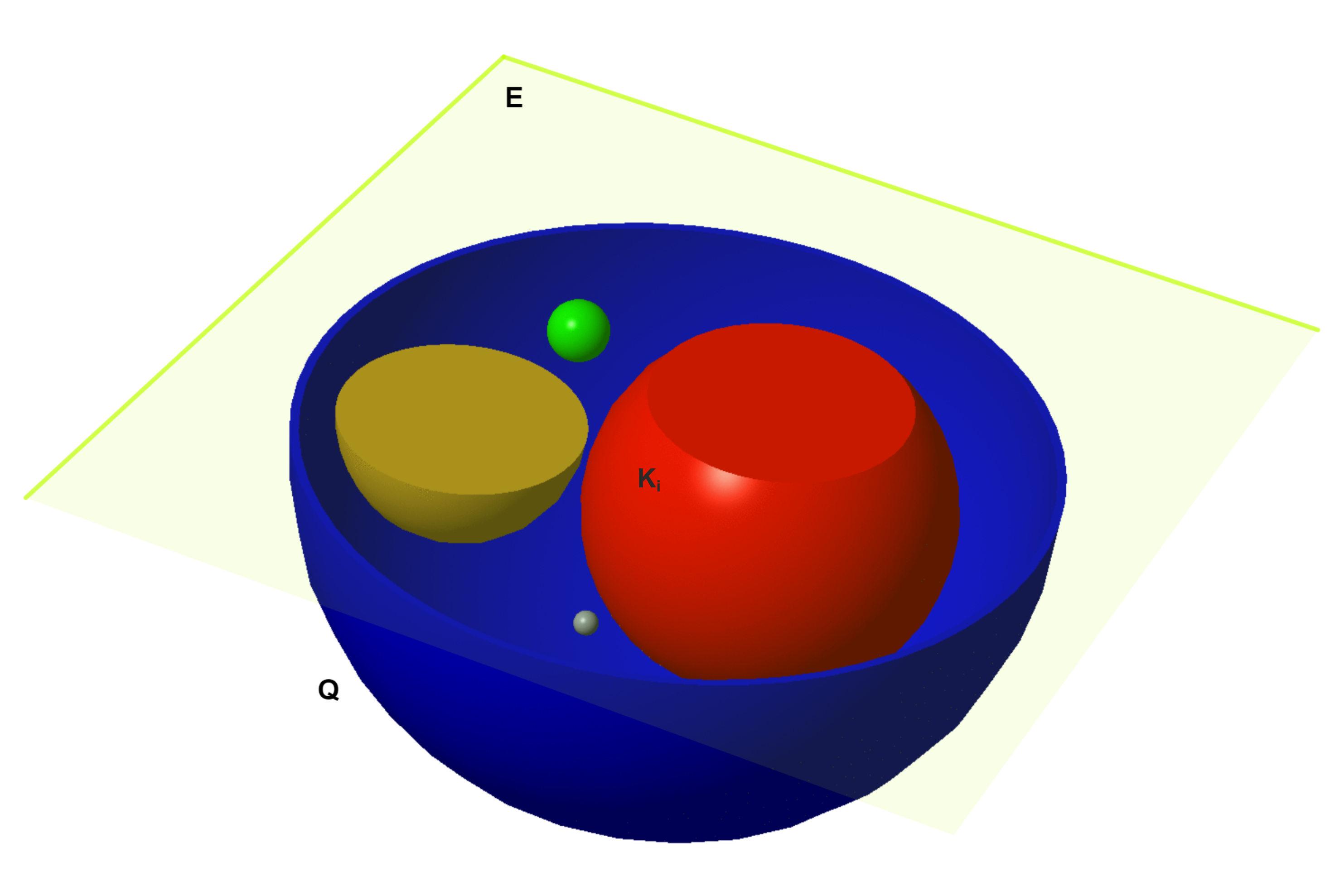}
    \caption{A section of a body $ \mathbf {Q} $ containing similar, spherical particles $ \mathbf {K_ {i}} $ by a random plane $ \mathbf {E} $.The shape of the body $ \mathbf {Q} $ can generally be different from the shape of the particles $ \mathbf {K_ {i}} $.}
    \label{rovF}
\end{figure}

\begin{exa}[Continuation of the spherical case example \ref{SCE}]

 Equation \eqref{RNR} has the form
\begin{equation}\label{SCeq}
\int\limits\limits\limits\limits\limits\limits\limits\limits\limits\limits\limits_{\left( \frac{\sigma}{\pi}\right)^{\frac{1}{2}}}^{\infty} \frac{H{\left(\lambda\right)} }{\left(\lambda^{2}\pi - \sigma \right)^{\frac{1}{2}}}\;\mathrm{d}\lambda = \frac{2\sqrt{\pi}}{\alpha}\,h{\left(\sigma\right)}.
\end{equation}

By using new variables $\pi\lambda^{2}= s$, $\mathfrak{H}\left(s\right) =  H \left( \left(\frac{s}{\pi}\right)^{\frac{1}{2}}\right)s^{-\frac{1}{2}}$ and $h_{1}{\left(\sigma\right)}= \frac{4\pi}{\alpha} h{\left(\sigma\right)}$ we transform equation
\eqref{SCeq} to

\begin{equation}
\int\limits\limits\limits\limits\limits\limits\limits\limits\limits\limits\limits_{\sigma}^{\infty} \frac{\mathfrak{H}{\left(s\right)} }{\left(s - \sigma \right)^{\frac{1}{2}}} \mathrm{d}s= h_{1}{\left(\sigma\right)},
\end{equation}

which is known Abel integral equation\footnote{Its solution has the form (if it is well defined)
\begin{equation}
\mathfrak{H}{\left(s\right)}= - \frac{1}{\pi}\int\limits\limits\limits\limits\limits\limits\limits\limits\limits\limits\limits_{s}^{\infty} \frac{h '_{1}{\left(\sigma\right)} \mathrm{d}\sigma}{\left(\sigma - s\right)^{\frac{1}{2}}}.
\end{equation}}, whereas its integral corresponds to ''1/2 integration'' but the unfolding problem is (moderately) ill posed. This yields the solution of original equation \eqref{SCeq} in the form

\begin{equation}\label{riesenie}
H_{pl}{\left( \lambda\right)}= - \frac{4\sqrt{\pi}}{\alpha}\lambda\int\limits\limits\limits\limits\limits\limits\limits\limits\limits\limits\limits_{\pi\lambda^{2}}^{\infty} \frac{h '{\left(\sigma\right)}}{\left(\sigma - \pi\lambda^{2}\right)^{\frac{1}{2}}} \mathrm{d}\sigma.
\end{equation}
Notice, that this form requires regularity of the input function $h$.
\end{exa}

\begin{rem}
Consider now that the section diameter $ r $ as random variable, whereas $\sigma = \pi r^{2}$. Let $ g \left( r \right)  = 2 \pi r h{\left(\pi r^{2}\right)}$ be a density\footnote{
In the geometric meaning $ g \left (r \right) \mathrm {d} r $, it expresses the number of particles per unit area in $ \mathbf {E} \cap \mathbf {Q} $, whose diameter in the cross-section is $( r , r + \mathrm {d} r )$.}, then 

\begin{equation}\label{ha}
      h{\left(\sigma\right)} = \frac{ g \left( r \right)}{2 \pi r}
\end{equation}
 and substituting \eqref{ha} into \eqref{riesenie} we obtain a solution
\begin{equation}
H{\left( \lambda\right)}= - \frac{2\lambda}{\pi\alpha} \int\limits\limits\limits\limits\limits\limits\limits\limits\limits\limits\limits_{\lambda}^{\infty}\left( \frac{ g \left( r \right)}{r}\right)^{'} \frac{\mathrm{d}r}{\left(r^{2} - \lambda^{2}\right)^{\frac{1}{2}}},
\end{equation}
which corresponds to  the solution of the well-known original Wicksell's corpuscule problem \cite{WICKSELL}, i.e. of the equation
\begin{equation}\label{rovnica}
    \frac{r}{1/\alpha}\int\limits_{r}^{R} \frac{ H\left(\lambda\right)}{\sqrt{\lambda^{2}-r^{2}}}\mathrm{d}\lambda = g\left(r\right), 
\end{equation}
with $R\in (0,\infty]$.
\end{rem}

\subsection{Correctness of a solution of \eqref{SCeq}, i.e. a spherical case}\label{AIR}

It can be verified that $ g $ is a density if it is $ H $. However the non-negativity of $ g $ does not guarantee the non-negativity of $ {H} $ on the entire interval. The $ g $ function must therefore satisfy the necessary condition that $ g \left (r \right)> 0 $ on the right neighborhood of zero. Notice that \cite{GorenfloVessel} they derive two conditions that must at least be met for the $ H $ function to be a density. We introduce here in notation of transformed $h.$

\begin{equation}
    \lim_{\lambda\to R} \lambda \int \limits _{\lambda}^{R}\frac{r h\left(\pi r^2\right)\mathrm{d}r}{\sqrt{r^{2}-\lambda^{2}}}= 0, \qquad \int \limits _{0}^{R}h\left(\pi r^2\right)\mathrm{d}r < \infty,
\end{equation}
whereas the first one is also necessary if $ R <\infty $ and the second one if $ R = \infty $ (it is even known that it is fulfilled if $r h(\pi r^2)\leq Ar^{-b}, ~A>0, b>1, r\geq r^*,$ so $r h(\pi r^2)\to 0$ for $r\to\infty$ sufficiently fast).
 


\subsection{Random line intersection}

Here we continue analogously for random line case. See Figure \ref{priamF}, where spherical case is illustrated. Similarly we have $\phi\in \mathcal{S}_{l_m}, ~l_m=\max l(\mathbf{G} \cap \mathbf{Q})$ and
\begin{equation}\nonumber
\int\limits\limits\limits\limits\limits\limits\limits\limits\limits\limits\limits\limits_{0}^{l_{m}}\phi\left(l\right)\dd l= 1, \qquad \int\limits_{0}^{l_{m}}l\,\phi\left(l\right)\dd l=  \frac{4\textit{V}}{\textit{F}}.
\end{equation}

Furthermore, the equation for random line intersection, see \cite{Santalo} or \cite{topics}, is
\begin{equation}\label{RNP}\tag{RL}
\boldmath{ \int\limits\limits\limits\limits\limits\limits\limits\limits\limits\limits\limits\limits_{l / l_{m}}^{\infty} \lambda \; \phi\left( \frac {l}{\lambda}\right)H\left(\lambda\right) \mathrm{d}\lambda = \frac{1}{\beta} \;h\left(l\right)},
\end{equation}
where $\beta=\frac{F}{4}$. Again relation for distribution in variable $l$ is given by $H_1(l)=H_{li}(l/l_m)/l_m$.

\begin{figure}[h!]
    \centering
    \includegraphics[width=0.8\textwidth]{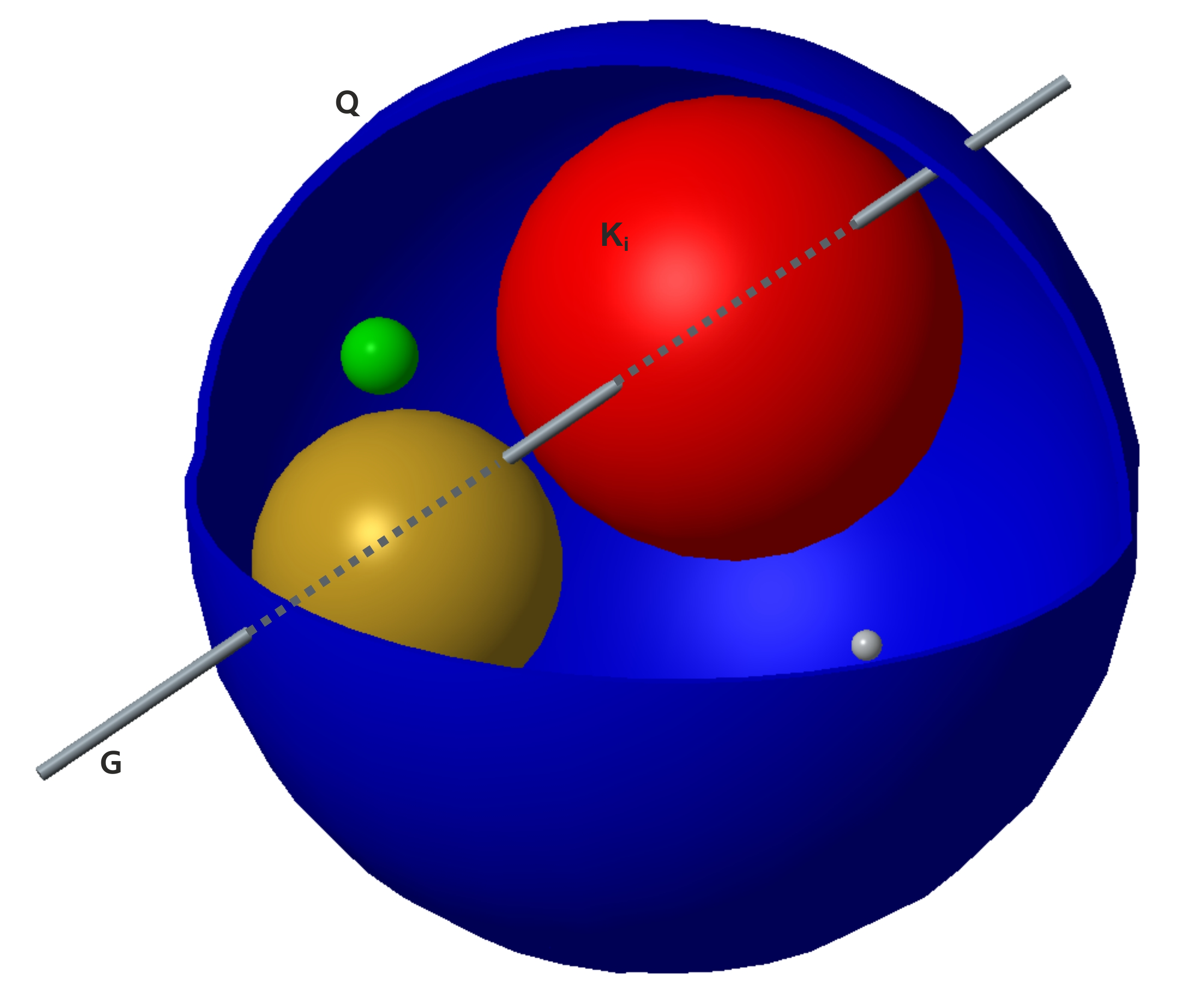}
       \caption{A section of a body $ \mathbf {Q} $ containing similar, spherical particles $ \mathbf {K_ {i}} $ by a random line $ \mathbf {G} $. The shape of the body $ \mathbf {Q} $ can generally be different from the shape of the particles $ \mathbf {K_ {i}} $.}
    \label{priamF}
\end{figure}

\begin{exa}[Spherical case]
We have $l_{m} = 2$ ($F = 4\pi$) and $\sigma(l)=\frac{l}{2}$,  which yields that equation \eqref{RNP} simplifies to

\begin{equation}\label{LCeq}
\int\limits\limits\limits\limits\limits\limits\limits\limits\limits\limits\limits\limits_{l / 2}^{\infty} H\left(\lambda\right) \mathrm{d}\lambda = \frac{2}{\beta\,l}\;h\left(l\right).
\end{equation}

Equation \eqref{LCeq} can be solved directly by differentiation with resulting form 

\begin{equation}\label{riesenie2}
H_{li}{\left( \lambda\right)}= - \frac{1}{\beta} \left( \frac{ h \left( 2\lambda \right)}{\lambda}\right)^{'}.
\end{equation}
\end{exa}

In further we set the notation
\begin{equation}\label{short}
\tilde{f}{\left(x\right)} = \left\{\begin{matrix}
f\left(x\right), & x\in \mathrm{supp}(f) 
\\ 0, & \mathrm{otherwise}.
\end{matrix}\right.
\end{equation}

{Exact solutions for similar convex bodies for plane sections}\label{rovina}

In this section we find exact solution for random plane case.
Consider here rewritten integral equation \eqref{RNR}
\begin{equation*}
\mathbf{K}\left[H\right]\coloneqq \alpha \int\limits_{\left(\frac{\sigma}{\sigma_{m}}\right)^\frac{1}{2}}^{\infty} \phi\left(\frac{\sigma}{\lambda^2}\right)\frac{H\left(\lambda\right)}{\lambda}\dd\lambda = h\left(\sigma\right), 
\end{equation*}

where $\sigma\in \left[0, \sigma_{m}\right],\sigma_{m} > 0, \lambda \geq\left(\frac{\sigma}{\sigma_{m}}\right)^\frac{1}{2}\geq 0$. We set in \eqref{RNR} $h(\sigma):= \sigma^{-s}, s \in \mathbb{C}$ yielding an equation
\begin{equation}\label{rnr2}
\mathbf{K}\left[H\right]  = \sigma^{-s}.
\end{equation}
We now seek a solution of equation \eqref{rnr2} in the form $H_{q}\left(\lambda\right)= k\left(q\right)\lambda^{q}$. We have 
\begin{equation}\label{vztah}
k\left(q\right) \alpha \int\limits_{\left(\frac{\sigma}{\sigma_{m}}\right)^\frac{1}{2}}^{\infty} \phi\left(\frac{\sigma}{\lambda^2}\right) \frac{\sigma^{s}}{\lambda^{1-q}} \dd\lambda = 1.
\end{equation}

This can be done by finding a relationship between $q$ and $s$. Transformation $\frac{\sigma}{\lambda^{2}}=z $ convert equation \eqref{vztah} to

\begin{equation}
\frac{k\left(q\right)}{2} \alpha \int\limits_{0}^{\sigma_{m}}\left. \phi\left( z\right) \frac{\sigma^{s-1}}{\lambda^{-q-2}}\right|_{\lambda^2=\frac{\sigma}{z}} \dd z = 1.
\end{equation}

Considering the form of Mellin transform one has to have ${s = -\frac{q}{2}}$.
We thus have $k\left(q\right)$ or $k(-2s)$ respectively. For simplicity we use
the notation $k(s)$ and similarly for function $H_{q}\left(\lambda\right)$.
Finally we have found $k$:
\begin{equation}
k\left(s\right) = \frac{2}{\alpha \int\limits_{0}^{\sigma_{m}} \phi\left( z\right) z^{s-1} \dd z}.
\end{equation}
Naturally \eqref{short} implies
$
\int\limits_{0}^{\sigma_{m}} \phi\left( z\right) z^{s-1} \dd z = \int\limits_{0}^{\infty} \tilde{\phi}\left( z\right) z^{s-1} \dd z,
$
therefore $k\left(s\right) = \frac{2}{\alpha\,\phi^*(s)}$
and 
\begin{equation}
H_{s}\left(\lambda\right) = \frac{2 \lambda^{-2s}}{\alpha\,\phi^*(s)}
\end{equation}
solves equation \eqref{rnr2}, whereas
$\phi^*(s):=\mathfrak{M}\left(\tilde{\phi}{\left(z\right)}\right).$

Now, we assume, which will be discussed later, that
\begin{equation}\nonumber
\exists \quad \mathfrak{M}\left(\tilde{h}\left(\sigma\right)\right) = h^{*}\left(s\right) \qquad \text{and} \qquad \exists \quad \mathfrak{M}^{-1}\left(h^{*}\left(s\right)\right) = \tilde{h}\left(\sigma\right).
\end{equation}
Since  $H_{s}\left(\lambda\right)$ solves equation \eqref{rnr2} we simply obtain
$\mathbf{K}\left[H_{s}\right]h^{*}\left(s\right) = \sigma^{-s} h^{*}\left(s\right),$ which gives
\begin{equation}
 \frac{1}{2\pi\textit{i}}\int\limits_{c - \textit{i}\infty}^{c + \textit{i}\infty}\mathbf{K}\left[H_{s}\right]h^{*}\left(s\right)\dd s = \underbrace{\frac{1}{2\pi\textit{i}}\int\limits_{c - \textit{i}\infty}^{c + \textit{i}\infty}\sigma^{-s} h^{*}\left(s\right) \dd s}_{\mathfrak{M}^{-1}\left(h^{*}\left(s\right)\right) = \tilde{h}\left(\sigma\right)} .
\end{equation}

Linearity and interchanging of order of integration yield

\begin{equation}
\alpha\int\limits_{\left(\frac{\sigma}{\sigma_{m}}\right)^\frac{1}{2}}^{\infty} \frac{\phi\left(\frac{\sigma}{\lambda^2}\right)}{\lambda} \left(\frac{1}{2\pi\textit{i}}\int\limits_{c - \textit{i}\infty}^{c + \textit{i}\infty}H_{s}\left(\lambda\right)h^{*}\left(s\right)\dd s \right)\dd\lambda = \tilde{h}\left(\sigma\right).
\end{equation}

Thus $\mathbf{K}\left[H\right]= \tilde{h}\left(\sigma\right)$ holds
for $H=\frac{1}{2\pi\textit{i}}\int\limits_{c - \textit{i}\infty}^{c + \textit{i}\infty}H_{s}\left(\lambda\right)h^{*}\left(s\right)\dd s$,
i.e. \begin{equation*}
\frac{{2}}{\alpha}\,\frac{1}{2\pi\textit{i}}\int\limits_{c - \textit{i}\infty}^{c + \textit{i}\infty}h^{*}\left(s\right)\frac{ \lambda^{-2s}}{\phi^{*}\left(s\right)}\dd s =  \frac{1}{\alpha}\,\mathfrak{M}^{-1}_{2c}\left(\frac{h^{*}\left(\frac{s}{2}\right)}{\phi^{*}\left(\frac{s}{2}\right)}\right)\left(\lambda\right),
\end{equation*}
 solves original equation \eqref{RNR}.

We denote this solution as
\begin{equation}\label{solRP}
\mathbb{H}_{pl}(\lambda) :=  \frac{1}{\alpha}\,\mathfrak{M}^{-1}_{2c}\left(\frac{h^{*}}{\phi^{*}}\left(\frac{s}{2}\right)\right)\left(\lambda\right).
\end{equation}

\section{Exact solutions for similar convex bodies for line sections}\label{priamka}

Now we focus on the random line case.
Consider here rewritten integral equation \eqref{RNP}
\begin{equation*}
\mathbf{L}\left[H\right]:=\beta\int\limits_{\frac{l}{l_{m}}}^{\infty} 
\phi\left(\frac{l}{\lambda}\right)\lambda H\left(\lambda\right)\dd\lambda
= h\left(l\right), 
\end{equation*}
where $l\in \left[0, l_{m}\right], l_{m} > 0, \lambda \geq\frac{l}{l_{m}}$. Again for $h\left(l\right):= l^{-s}, s \in \mathbb{C}$ in \eqref{RNP} we have
\begin{equation}\label{rnp2}
\mathbf{L}\left[H\right] = l^{-s},
\end{equation}
with the assumption that its solution has the form $H_{q}\left(\lambda\right)= k\left(q\right)\lambda^{q}$. Similarly we obtain

\begin{equation}
\int\limits_{\frac{l}{l_{m}}}^{\infty}\phi\left(\frac{l}{\lambda}\right)\frac{l^{s}}{\lambda^{-q-1}} \dd\lambda = 1
\end{equation}
and for $\frac{l}{\lambda}=z$ 
\begin{equation}
k\left(q\right)\int\limits_{0}^{l_{m}} \left.\phi\left( z\right) \frac{l^{s-1}}{\lambda^{-q-3}} \right|_{\lambda=\frac{l}{z}}\dd z = 1.\end{equation}
Once we set ${-2 - q= s}$, we find $k:$
\begin{equation}
k\left(s\right) =\left(\beta\ \int\limits_{0}^{l_{m}} \phi\left( z\right) z^{s-1} \dd z\right)^{-1}=\frac{1}{\beta\ \phi^*(s)},
\end{equation}
and the solution
\begin{equation}
H_{s}\left(\lambda\right)  = \frac{\lambda^{-s-2}}{\beta\ \phi^*(s)}
\end{equation}
of equation \eqref{rnp2}.  Similarly as before 
$ \mathbf{L}\left[H_{s}\right]h^{*}\left(s\right) = l^{-s} h^{*}\left(s\right)  $
is just
\begin{equation}
 \frac{1}{2\pi\textit{i}}\int\limits_{c - \textit{i}\infty}^{c + \textit{i}\infty}\mathbf{L}\left[H_{s}\right] h^{*}\left(s\right)\dd s = \underbrace{\frac{1}{2\pi\textit{i}}\int\limits_{c - \textit{i}\infty}^{c + \textit{i}\infty}h^{*}\left(s\right)l^{-s}  \dd s}_{\mathfrak{M}^{-1}\left(h^{*}\left(s\right)\right) = \tilde{h}\left(l\right)}
\end{equation}
and

\begin{equation}
\beta\ \int\limits_{\frac{l}{l_{m}}}^{\infty} \lambda\phi\left(\frac{l}{\lambda}\right) \left(\frac{1}{2\pi\textit{i}}\int\limits_{c - \textit{i}\infty}^{c + \textit{i}\infty}H_{s}\left(\lambda\right)h^{*}\left(s\right) \dd s\right)\dd\lambda= \tilde{h}\left(l\right).
\end{equation}

Finally 
\begin{equation}\label{solRL}
 \mathbb{H}_{li}\left(\lambda\right) := \frac{1}{\beta}\frac{\lambda^{-2}}{2\pi\textit{i}}\int\limits_{c - \textit{i}\infty}^{c + \textit{i}\infty}\frac{h^{*}\left(s\right)}{\phi^{*}\left(s\right)} \lambda^{-s}\dd s 
= \frac{1}{\beta\,\lambda^{2}}\, \mathfrak{M}^{-1}_c\left(\frac{h^{*}\left(s\right)}{\phi^{*}\left(s\right)}\right)\left(\lambda\right)
\end{equation}

solves original equation \eqref{RNP}.

\section{Existence of the solutions of equations
\eqref{RNR} and \eqref{RNP}}\label{exist}

Here we will consider the question of the existence of solutions of equations \eqref {RNR} and \eqref {RNP}. In the context of integral equations, this is a difficult task, mainly because they are singular (albeit linear) equations. Notice that there is also statistical interpretation of their derived form. They have the form of inverse Mellin transformation of the ratio of $\frac{s}{2}-1$ moments and $s-1$ moments respectively, of random variables corresponding to the densities of $ f $ and $ \phi $. Of course, the proportion of moments does not have to be a moment, so we cannot talk about a direct relationship to the new density. Indeed, for both equations, the non-negativity of the functions (inputs) $ h $ and $ \phi $ does not guarantee the non-negativity of the $ H $ solution. Also, it can not be guaranteed that $ H \in L (\mathbb {R} _ +) $. Thus, none of the basic density properties need not be met even if the inputs are densities. For general inputs $ h, \phi $, determining the necessary or sufficient conditions is a very difficult problem (non-negativity, but also integrability property of the solution). Indeed, just look at the results for the spherical ones listed in \ref {AIR}. Nevertheless, we give here a partial answer. In further we denote as 
$\gamma_F:=\inf\{\alpha^*: F=\mathcal{O}(x^{-\alpha^*}), ~x\to0^+\}$ and $\gamma:=\max\{\gamma_h,\gamma_\phi\}.$

\begin{prop}[Random lines]
\noindent  \phantom{c}
\begin{enumerate}[I)]
    \item Suppose that $h, ~\phi \in \mathcal{S}_{l_m}$ and  $h^*(\mu+\mathrm{i}\,\cdot), ~\frac{h^*}{\phi^*}(\mu+\mathrm{i}\,\cdot)\in L(\mathbb{R})$ for some $\mu>\gamma.$  Then there exists a unique solution of equation \eqref{RNP} and has the form \eqref{solRL}.
\item If furthermore $K>0$: ~$\left|\frac{h^{*}}{\phi^{*}}\left(s\right)\right|\leq K|s|^{-2}$, then $\mathbb{H}_{pl}\in C(\mathbb{R}_+)$ and if $\frac{h^*}{\phi^*}(\mu+\mathrm{i}\,\cdot)\in L^p(\mathbb{R}), ~p\in(1,2]$ for some $\mu>\gamma$, then $\lambda^2\,\mathbb{H}_{pl}(\lambda)\in L^{p/(p-1)}_{\{\mu\}}(\mathbb{R}_+).$
\end{enumerate}
\end{prop}

\begin{proof}
Since $h, ~\phi \in \mathcal{S}_{\sigma_m}$, from Lemma \ref{lem1} it follows that there exist $\gamma_h<1$
and $\gamma_\phi<1$ so, that $h\in L_{(\gamma_h,\infty)}(\mathbb{R}_+), ~\phi\in L_{(\gamma_\phi,\infty)}(\mathbb{R}_+)$. Thus from Theorem \ref{ThH} we have $h^*\in \mathcal{H}(\mathrm{St}(\gamma_h, \infty))$ and $\phi^*\in \mathcal{H}(\mathrm{St}(\gamma_\phi, \infty))$. This means, that for 
$\gamma,$ which is less than 1 therefore
 $h^*, ~\phi^*\in \mathcal{H}(\mathrm{St}(\gamma, \infty))$.
Again from Lemma \ref{lem1} we have that, $\phi^*$ is nonzero and thus the quotient  $\frac{h^*}{\phi^*}\in \mathcal{H}(\mathrm{St}(\gamma, \infty)),$ which means that the necessary condition for this function to be Mellin's image is fulfilled. Since $\frac{h^*}{\phi^*}(\mu+\mathrm{i}\,\cdot)\in L(\mathbb{R})$, from Theorem \ref{MTT} it follows that $\mathbb{H}_{pl}$ is well defined. Furthermore, it follows from Fubini's theorem and transformation theorem that
$$\mathbf{L}\left[\mathbb{H}_{pl}\right]=\int_\mathbb{R}\int_{\frac{l}{l_{m}}}^\infty\frac{\lambda^{-1-\mu-\mathrm{i}\,\nu}}{2\pi}\phi\left(\frac{l}{\lambda}\right)\frac{h^*}{\phi^*}(\mu+\mathrm{i}\,\nu)\dd\lambda\dd\nu=$$
$$=\frac{1}{2\pi}\int_\mathbb{R}l^{-\mu-\mathrm{i}\,\nu}\,\frac{h^*}{\phi^*}(\mu+\mathrm{i}\,\nu)\int_0^{l_m}
z^{\mu+\mathrm{i}\,\nu-1}\phi(z)\dd z\dd\nu =
\frac{l^{-\mu}}{2\pi}\int_\mathbb{R}h^*(\mu+\mathrm{i}\,\nu)\,l^{-\mathrm{i}\,\nu}\dd\nu.$$
Since $h^*(\mu+\mathrm{i}\,\cdot)\in L(\mathbb{R})$, from the inverse theorem \ref {VoINV} we get that $\mathbf{L}\left[\mathbb{H}_{pl}\right]=h$
a.e. This guarantees the existence since a.e. equal functions $ h_1 , h_2 $  have the same Mellin's image. Now, if $\mathbb{H}_{pl}^1$ and $\mathbb{H}_{pl}^2$ solves given equation, then equality $0=\mathbf{L}\left[\mathbb{H}_{pl}^1\right]-\mathbf{L}\left[\mathbb{H}_{pl}^2\right]=\mathbf{L}\left[\mathbb{H}_{pl}^1-\mathbb{H}_{pl}^2\right]=\beta\ \int\limits_{\frac{l}{l_{m}}}^{\infty} \lambda\phi\left(\frac{l}{\lambda}\right) (\mathfrak{H}_1(\lambda)-\mathfrak{H}_2(\lambda))\dd\lambda$ implies uniqueness, since $\lambda\phi\left(\frac{l}{\lambda}\right)>0$ a.e.
The second part of the proposition follows from Theorems \ref {ohran} and \ref {LP}.
\end{proof}

The proof of the following proposition is identical to the previous one. 

\begin{prop}[Random planes]
\noindent\phantom{c}
\begin{enumerate}[I)]
    \item Suppose that $h, ~\phi \in \mathcal{S}_{l_m}$ and  $h^*(\mu+\mathrm{i}\,\cdot), ~\frac{h^*}{\phi^*}(\mu+\mathrm{i}\,\cdot)\in L(\mathbb{R})$ for some $\mu>\gamma.$  Then there is a unique solution of equation \eqref {RNR}, having the form \eqref{solRP}
\item If moreover there exists $K>0$: ~$\left|\frac{h^{*}}{\phi^{*}}\left(\frac{s}{2}\right)\right|\leq K|s|^{-2}$, then $\mathbb{H}_{li}\in C(\mathbb{R}_+)$ and if $\frac{h^*}{\phi^*}(\mu+\mathrm{i}\,\cdot)\in L^p(\mathbb{R}), ~p\in(1,2]$ for some $\mu>\gamma$, then $\mathbb{H}_{li}\in L^{p/(p-1)}_{\{\mu\}}(\mathbb{R}_+).$
\end{enumerate}
\end{prop}

Notice that for $p$ we have sufficient conditions for the solutions $\mathbb{H}_{pl}$ and $\lambda^2\,\mathbb{H}_{pl}$ to be in $L^2$. It is good to realize that we can say more. Indeed, independence of Mellin's inversion on $ \mu $ (for which $ h^* / \phi^* $ can be integrated) must necessarily follow from the uniqueness of Mellin transform. Thus, if it fulfills the condition of uniform convergence \eqref{NP}, then $\lambda^2\,\mathbb{H}_{li}(\lambda)$ and $\mathbb{H}_{pl}(\lambda)$ are in $L_{(\alpha,\beta)}(\mathbb{R}_+)$ for some interval $(\alpha,\beta).$

\section{Random plane examples}\label{RPex}

 Recall that we have two formulas, that give us solution of the integral equation \eqref{RNR}, in spherical case, the classical one \eqref{riesenie} and our formula \eqref{solRP}. In this section we show that the former one can not be used when violating regularity of the right hand site. 

\begin{exa}[Nearly spherical case]
It is a natural generalization of spherical case (i.e. when $p=1/2$) with the density kernel
\begin{equation}\label{NSC}
\phi{\left(\sigma\right)} = \left\{\begin{matrix}
\frac{1-p}{\sigma_m^{1-p}(\sigma_m-\sigma)^p}, & \sigma \in \left[0, \sigma_m\right] 
\\ 0, & \sigma > \sigma_m
\end{matrix}\right..
\end{equation}
where $\sigma_m\geq F/4$ and $p=\frac{M \sigma_m}{2\pi V}\leq 1/2$. This leads to generalized Abel integral equation and similar type of solution of the form
\begin{equation}\label{SEX3}
    H(\lambda)=\zeta(\lambda;\sigma_m,p,M)\,\int_{\pi\lambda^{2}}^{\infty} \frac{h '{\left(\sigma\right)}}{\left(\sigma - \sigma_m\lambda^{2}\right)^{p}} \mathrm{d}\sigma,
\end{equation}
see e.g. \cite{Santalo}. Again regularity condition on $h$ is needed. See Figure \ref{EX3} for case of $p=1/4, K=2, \sigma_m=\pi$.

\begin{figure}[h!]
     \centering
     \includegraphics[width=0.5\textwidth]{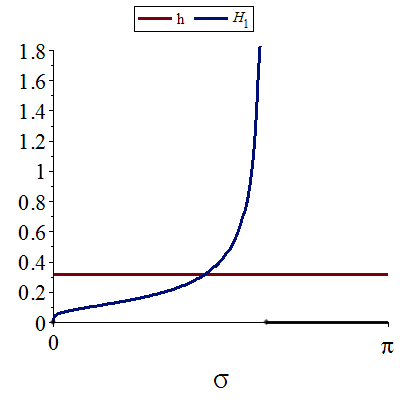}
     \caption{Graph of the uniform density $h$ and a solution $H_1$ given by \eqref{SEX3}.}
     \label{EX3}
 \end{figure}

\end{exa}
 
 The following example shows that the original formula derived formally (Abel integral equation) is not always applicable. Recall that Beta function is defined as
${\displaystyle \mathrm {B} (x,y)=\int _{0}^{1}t^{x-1}(1-t)^{y-1}\,dt}, ~\Re(x), \Re(y)>0$
and ${\displaystyle \mathrm {B} (x,y)={\frac {\Gamma (x)\,\Gamma (y)}{\Gamma (x+y)}}},$ ~${\displaystyle \Gamma (z)={\frac {\Gamma (z+1)}{z}}}, ~\Re(z)>0$.

\begin{exa}[Spherical case, uniform distribution]

Consider the equation \eqref{RNR} and uniform density $h\left(\sigma\right)$ in the form

\begin{equation}\nonumber
h{\left(\sigma\right)} = \left\{\begin{matrix}
\frac{1}{\pi}, & \sigma \in \left[0, \pi\right] 
\\ 0, & \sigma > \pi
\end{matrix}\right..
\end{equation}

\begin{figure}[h!]
     \centering
     \includegraphics[width=0.5\textwidth]{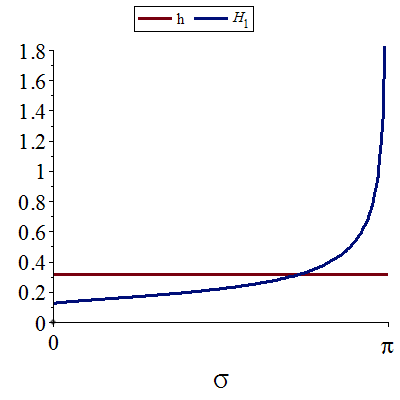}
     \caption{Graph of the uniform density $h$ and a solution $H_1$ given by \eqref{SEX1}.}
     \label{EX1}
 \end{figure}

First we find $h^{*}{\left(s\right)}$ and $\phi^{*}{\left(s\right)}$:
\begin{multline*}
     h^{*}{\left(s\right)} = \int\limits_{0}^{\infty} h\left(\sigma\right) \sigma^{s-1} \dd \sigma =  \frac{1}{\pi} \int\limits_{0}^{\pi} \sigma^{s-1} \dd \sigma  = \frac{\pi^{s-1}}{s}, \quad {~\Re(s)> 0},
\end{multline*}
\begin{equation*}
    \phi^{*}{\left(s\right)} = \int\limits_{0}^{\infty} \phi\left(\sigma\right) \sigma^{s-1} \dd \sigma = \int_0^\pi \frac{\sigma^{s-1}}{2\sqrt{\pi(\pi-\sigma)}}\dd \sigma =\frac{{\pi}^{s-1}}{2}\int_{0}^{1}\!{\frac {{t}^{s-1}}{\sqrt {1-t}}}\,{\rm d}t=
    \end{equation*}
\begin{equation*}
    =\frac{{\pi}^{s-1}}{2} B(s,1/2)=
    \frac{{\pi}^{s-1} \Gamma(s) \Gamma(1/2)}{2\Gamma(s+1/2)}=
    \frac{{\pi}^{s-1/2}\,\Gamma(s) }{2\Gamma(s+1/2)}, \quad  {~\Re(s)> 0}.
\end{equation*}

We thus have
\begin{equation*}
   { \frac{h^{*}{\left(s/2\right)}}{\phi^{*}{\left(s/2\right)}} = {\frac {4 \Gamma  \left( 1/2+s/2 \right) }{\sqrt {\pi }\,s\,\Gamma  \left(s/2 \right) }}, \quad ~\Re(s)> 0.}
\end{equation*}
   { with inverse Mellin transform, see \cite[II.5.35]{oberhettinger}) has the solution}
\begin{equation}\label{SEX1}
    \mathbb{H}_{pl}{\left( \lambda\right)}=\frac{1}{\alpha}{ \  \mathfrak{M}^{-1}_c\left(\frac{h^{*}}{\phi^{*}}\left(\frac{s}{2}\right)\right)\left(\lambda\right) {=} \frac{3{\lambda}}{4 \sqrt{1 - \lambda}}, ~\lambda\in[0,1)}
\end{equation}
see Figure \ref{EX1}. We have to emphasize that uniform case yields zero solution $ H_{pl}{\left( \lambda\right)}= 0$ a.e., which obviously does not fulfill equation \eqref{RNR}. Notice moreover that $\mathbb{H}_{pl}\in L^1(\mathbb{R}^+) $ and $\alpha$
is for us here normalizing constant.
\end{exa}

\begin{exa}[Nearly spherical case, uniform distribution]
For uniform density
\begin{equation}\nonumber
h{\left(\sigma\right)} = \left\{\begin{matrix}
\frac{1}{K}, & \sigma \in \left[0, K\right] 
\\ 0, & \sigma > K
\end{matrix}\right., K>0
\end{equation}
and density given by \eqref{NSC} one can show that the solution has normalized form (i.e. it is a density)
$$\mathbb{H}_{sp}(\lambda)={\frac {{{\sigma}_{{m}}}^{\frac{3}{2}-p}\sqrt {\pi }}{\Gamma 
 \left( p \right) \Gamma  \left( 3/2-p \right) \sqrt {K}} \left( {\frac {K}{{
\lambda}^{2}}}-{\sigma}_{{m}} \right) ^{p-1}},  ~~\lambda<\sqrt{K/\sigma_m}.$$
\end{exa}
 
\section{Random line examples}\label{RLex}

In this section we give several examples of solutions of the integral equation \eqref{RNP}. Recall that we have two formulas in spherical case, i.e. classical \eqref{riesenie2} and our formula \eqref{solRL}. In this case they coincide.

\begin{exa}[Spherical case, triangle distribution]
This example shows that the support of resulting density need not be full original set. Consider equation \eqref{RNP} the density $h\left(l\right)$ in the form
\begin{equation}\nonumber
h{\left(l\right)} = \left\{\begin{matrix}
1 - \left| l - 1 \right|, & l \in \left[0, 2\right] 
\\ 0, & l > 2
\end{matrix}\right..
\end{equation}

\begin{figure}[h!]
     \centering
     \includegraphics[width=0.5\textwidth]{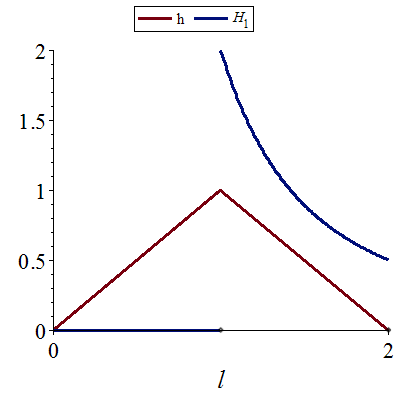}
     \caption{Graph of the uniform density $h$ and a solution $H_1$ given by \eqref{SEX2}.}
     \label{EX2}
 \end{figure}

One can directly obtain $h^{*}{\left(s\right)}$ and $\phi^{*}{\left(s\right)}$:
\begin{multline*}
     h^{*}{\left(s\right)} = \int\limits_{0}^{\infty} h\left(l\right) l^{s-1} \dd l =  \int\limits_{0}^{1} l \,l^{s-1} \dd l + \int\limits_{1}^{2}  \left(2- l\right) l^{s-1} \dd l = \frac{2^{1+s}-2}{s\left(1+s\right)}, ~\Re(s)> 0.
\end{multline*}
\begin{equation*}
    \phi^{*}{\left(s\right)} = \int\limits_{0}^{\infty}\phi\left(l\right) l^{s-1} \dd l = \frac{1}{2}\int\limits_{0}^{2} l^{s} \dd l =
     \frac{2^{s}}{s+1}, \quad ~\Re(s)> -1.
\end{equation*}
Therefore
\begin{equation*}
    \frac{h^{*}{\left(s\right)}}{\phi^{*}{\left(s\right)}} = \frac{2 - 2^{1 - s}}{s},  ~\Re(s)> 0
\end{equation*}
and inverse Mellin transform (see \cite[II. 2.2.1.]{oberhettinger}) has the form
\begin{equation}\nonumber
   \mathfrak{M}^{-1}_c\left(\frac{h^{*}\left(s\right)}{\phi^{*}\left(s\right)}\right)\left(\lambda\right) =  \mathfrak{M}^{-1}_c\left(\frac{2}{s} - \frac{2^{1 - s}}{s}\right)\left(\lambda\right) {=}\left\{\begin{matrix}
2, & \lambda \in \left[\frac{1}{2}, 1\right] 
\\ 0, & otherwise
\end{matrix}\right..
\end{equation}
Finally
\begin{equation}\label{SEX2}
    \mathbb{H}_{li}{\left( \lambda\right)} = \left\{\begin{matrix}
\frac{2}{\beta\,\lambda^{2}}, & \lambda \in \left[\frac{1}{2}, 1\right] 
\\ 0, & otherwise
\end{matrix}\right.,
\end{equation}
see Figure \ref{EX2}. Again $\beta$ here is normalizing constant. This is in a coincidence with $H_{li}{\left( \lambda\right)}$ given by the classical formula.  
 \end{exa}

\begin{exa}[Spherical case, quadratic distribution]
Again consider equation \eqref{RNP} and right hand site $h\left(l\right)$ in the form

\begin{equation}\nonumber
h{\left(l\right)} = \left\{\begin{matrix}
\frac{3}{8}\left(2 - l\right)^{2}, & l \in \left[0, 2\right] 
\\ 0, & l > 2
\end{matrix}\right..
\end{equation}

We have
\begin{multline*}
     h^{*}{\left(s\right)} = \int\limits_{0}^{\infty} h\left(l\right) l^{s-1} \dd l =  \int\limits_{0}^{2} \frac{3}{8} \left(2 - l\right)^{2} l^{s-1} \dd l  
     = \frac{3\,2^{s}}{s\left(s^{2}+ 3s + 2\right)}, \quad ~\Re(s)>0.
\end{multline*}
 and again
 \begin{equation*}
    \phi^{*}{\left(s\right)}  = \frac{2^{s}}{s+1}, \quad ~\Re(s)> -1.
\end{equation*}
Thus
\begin{equation*}
    \frac{h^{*}{\left(s\right)}}{\phi^{*}{\left(s\right)}} = \frac{3}{s^{2} + 2s}, \quad ~\Re(s)> 0.
\end{equation*}
with inverse Mellin transform (see \cite[II.2.9.]{oberhettinger})  of the form
\begin{equation}\nonumber
   \mathfrak{M}^{-1}_c\left(\frac{h^{*}\left(s\right)}{\phi^{*}\left(s\right)}\right)\left(\lambda\right) =  \mathfrak{M}^{-1}_c\left(\frac{3}{s^{2} + 2s}\right) \left(\lambda\right)  {=} \frac{3}{2}\left(1-\lambda^{2}\right), ~\lambda\in[0,1] 
\end{equation}
This yields
\begin{equation*}
    \mathbb{H}_{li}{\left( \lambda\right)} = \left\{\begin{matrix}
\frac{3}{2\beta\,\lambda^{2}}\left(1-\lambda^{2}\right), & \lambda \in \left(0, 1\right]
\\ 0, & otherwise
\end{matrix}\right..
\end{equation*}

Again $H_{li}{\left( \lambda\right)} =\mathbb{H}_{li}{\left( \lambda\right)}$ a.e. but notice that this solution is not integrable (thus it can not be normalized in standard sense).
\end{exa}

\section*{Acknowledgements}
This work was supported by the Slovak Research and Development
Agency under the contracts No. APVV-16-0337 and APVV-17-0568.

\bibliographystyle{imsart-nameyear}
\bibliography{references}

\begin{thebibliography}{13}

\bibitem[\protect\citeauthoryear{Baddeley and Jensen}{2004}]{SfS}
\begin{bbook}[author]
\bauthor{\bsnm{Baddeley},~\bfnm{A.}\binits{A.}} \AND
  \bauthor{\bsnm{Jensen},~\bfnm{E.~B.~V.}\binits{E.~B.~V.}}
(\byear{2004}).
\btitle{Stereology for Statisticians}.
\bseries{Chapman \& Hall/CRC Monographs on Statistics \& Applied Probability}.
\bpublisher{Taylor \& Francis}.
\end{bbook}
\endbibitem

\bibitem[\protect\citeauthoryear{Bertrand and Ovarlez}{2018}]{poularikas}
\begin{binbook}[author]
\bauthor{\bsnm{Bertrand},~\bfnm{J.~Bertrand~P.}\binits{J.~B.~P.}} \AND
  \bauthor{\bsnm{Ovarlez},~\bfnm{J.~P.}\binits{J.~P.}}
(\byear{2018}).
\btitle{The Mellin Transform}.
In \bbooktitle{Transforms and Applications Handbook}.
\bseries{Electrical Engineering Handbook}
\bchapter{12}.
\bpublisher{CRC Press}.
\end{binbook}
\endbibitem

\bibitem[\protect\citeauthoryear{Butzer and Jansche}{1997}]{DAMT}
\begin{barticle}[author]
\bauthor{\bsnm{Butzer},~\bfnm{Paul~L.}\binits{P.~L.}} \AND
  \bauthor{\bsnm{Jansche},~\bfnm{Stefan}\binits{S.}}
(\byear{1997}).
\btitle{A direct approach to the mellin transform}.
\bjournal{Journal of Fourier Analysis and Applications}
\bvolume{3}
\bpages{325--376}.
\bdoi{10.1007/BF02649101}
\end{barticle}
\endbibitem

\bibitem[\protect\citeauthoryear{De-Lin}{1994}]{topics}
\begin{bbook}[author]
\bauthor{\bsnm{De-Lin},~\bfnm{Ren}\binits{R.}}
(\byear{1994}).
\btitle{Topics in integral geometry}
\bvolume{19}.
\bpublisher{World Scientific Publishing Company}.
\end{bbook}
\endbibitem

\bibitem[\protect\citeauthoryear{Enderlein}{1965}]{Ken}
\begin{barticle}[author]
\bauthor{\bsnm{Enderlein},~\bfnm{G.}\binits{G.}}
(\byear{1965}).
\btitle{Kendall, M. G., and P. A. P. Moran: Geometrical Probability. Griffin,
  London 1963; 125 S., Preis 28 s}.
\bjournal{Biometrische Zeitschrift}
\bvolume{7}
\bpages{208-208}.
\bdoi{10.1002/bimj.19650070322}
\end{barticle}
\endbibitem

\bibitem[\protect\citeauthoryear{Flajolet, Gourdon and Dumas}{1995}]{MTA}
\begin{barticle}[author]
\bauthor{\bsnm{Flajolet},~\bfnm{Philippe}\binits{P.}},
  \bauthor{\bsnm{Gourdon},~\bfnm{Xavier}\binits{X.}} \AND
  \bauthor{\bsnm{Dumas},~\bfnm{Philippe}\binits{P.}}
(\byear{1995}).
\btitle{Mellin transforms and asymptotics: Harmonic sums}.
\bjournal{Theoretical computer science}
\bvolume{144}
\bpages{3--58}.
\end{barticle}
\endbibitem

\bibitem[\protect\citeauthoryear{Gabutti and Sacripante}{1991}]{GABUTTI1991191}
\begin{barticle}[author]
\bauthor{\bsnm{Gabutti},~\bfnm{B.}\binits{B.}} \AND
  \bauthor{\bsnm{Sacripante},~\bfnm{L.}\binits{L.}}
(\byear{1991}).
\btitle{Numerical inversion of the Mellin transform by accelerated series of
  Laguerre polynomials}.
\bjournal{Journal of Computational and Applied Mathematics}
\bvolume{34}
\bpages{191 - 200}.
\bdoi{https://doi.org/10.1016/0377-0427(91)90041-H}
\end{barticle}
\endbibitem

\bibitem[\protect\citeauthoryear{Gorenflo and Vessella}{1980}]{GorenfloVessel}
\begin{barticle}[author]
\bauthor{\bsnm{Gorenflo},~\bfnm{Rudolf}\binits{R.}} \AND
  \bauthor{\bsnm{Vessella},~\bfnm{Sergio}\binits{S.}}
(\byear{1980}).
\btitle{Abel integral equations(analysis and applications)}.
\bjournal{Lecture Notes in Mathematics}.
\end{barticle}
\endbibitem

\bibitem[\protect\citeauthoryear{Oberhettinger}{2012}]{oberhettinger}
\begin{bbook}[author]
\bauthor{\bsnm{Oberhettinger},~\bfnm{F.}\binits{F.}}
(\byear{2012}).
\btitle{Tables of Mellin Transforms}.
\bpublisher{Springer Berlin Heidelberg}.
\end{bbook}
\endbibitem

\bibitem[\protect\citeauthoryear{{Santal\'o}}{1976}]{Santalo}
\begin{bbook}[author]
\bauthor{\bsnm{{Santal\'o}},~\bfnm{Luis}\binits{L.}}
(\byear{1976}).
\btitle{{Integral geometry and geometric probability. With a foreword by Mark
  Kac.}},
\bedition{1.} ed.
\bpublisher{Addison‐Wesley Publishing Company}.
\end{bbook}
\endbibitem

\bibitem[\protect\citeauthoryear{Tsamasphyros and
  Theocaris}{1976}]{Tsamasphyros1976}
\begin{barticle}[author]
\bauthor{\bsnm{Tsamasphyros},~\bfnm{G.}\binits{G.}} \AND
  \bauthor{\bsnm{Theocaris},~\bfnm{P.~S.}\binits{P.~S.}}
(\byear{1976}).
\btitle{Numerical inversion of Mellin transforms}.
\bjournal{BIT Numerical Mathematics}
\bvolume{16}
\bpages{313--321}.
\bdoi{10.1007/BF01932274}
\end{barticle}
\endbibitem

\bibitem[\protect\citeauthoryear{Wicksell}{1925}]{WICKSELL}
\begin{barticle}[author]
\bauthor{\bsnm{Wicksell},~\bfnm{S.~D.}\binits{S.~D.}}
(\byear{1925}).
\btitle{{The corpuscle problem. A mathematical study of a biometric problem}}.
\bjournal{Biometrika}
\bvolume{17}
\bpages{84-99}.
\bdoi{10.1093/biomet/17.1-2.84}
\end{barticle}
\endbibitem

\bibitem[\protect\citeauthoryear{Yakubovich}{1996}]{yakubovich}
\begin{bbook}[author]
\bauthor{\bsnm{Yakubovich},~\bfnm{S.~B.}\binits{S.~B.}}
(\byear{1996}).
\btitle{Index Transforms}.
\bpublisher{World Scientific}.
\end{bbook}
\endbibitem

\end{thebibliography}

\appendix

\section{The Method of Model Solutions}\label{MMR}

Consider a linear equation in the operator form
\begin{equation}\label{mmr1}
\mathbf{O} \left[y\right] = f\left(x\right), 
\end{equation}
where $\mathbf{O}$ linear (integral) operator\footnote{ $\mathbf{O}$ has to be independent of $\lambda.$}, $y\left(x\right)$ is an unknown function, and $f\left(x\right)$ is a known function. Let us have a (non constant) test solution
\begin{equation}\label{mmr2}
y_{0} = y_{0}\left(x, \lambda\right), 
\end{equation}
depending on an auxiliary parameter $\lambda$. The right-hand side that corresponds to the test solution \eqref{mmr2} is
\begin{equation}\label{mmr3}
f_{0}\left(x, \lambda\right) = \mathbf{O}  \left[y_{0}\right].
\end{equation}
  Let us multiply equation \eqref{mmr3} by some function $\varphi \left(\lambda\right)$ and integrate it with respect to $\lambda$ over an suitable interval $\left[a , b\right]$. Assumption of the interchange of the order of integration yields
\begin{equation}\label{mmr4}
\mathbf{O} \left[y_{\varphi}\right] = f_{\varphi}\left(x\right),
\end{equation}
with
\begin{equation}\label{mmr5}
y_{\varphi}\left(x\right) = \int\limits_{a}^{b} y_{0}\left(x, \lambda\right)\varphi \left(\lambda\right) \dd \lambda, \qquad f_{\varphi}\left(x\right) = \int\limits_{a}^{b} f_{0}\left(x, \lambda\right)\varphi \left(\lambda\right) \dd \lambda.
\end{equation}

Then from $\eqref{mmr4}$ and $\eqref{mmr5}$
follows, for the right-hand side  $f = f_{\varphi}\left(x\right)$, that the function $y = y_{\varphi}\left(x\right)$ is a solution of the original equation $\eqref{mmr1}$. Here the main
problem is how to choose a function $\varphi \left(\lambda\right)$ to obtain a given function $f_{\varphi}\left(x\right)$. This can be overcome by finding a test solution $Y\left(x, \lambda\right)$ (called a {\it model solution}) such that the right-hand side of $\eqref{mmr1}$ is the kernel of a known inverse integral transform.

\section{Mellin transform}\label{MT}

The following section discusses Mellin integral transform and its properties. In the text we rely mainly on works  \cite{DAMT} and \cite{MTA}. In our problem of solving \eqref{RNR} and \eqref{RNP}, using Mellin transform seems paradoxically easier that using Laplace or Fourier. 

Denote as $\mathrm{St}(\alpha, \beta)$ an open strip of complex number $s = \mu + \mathrm{i}\,\nu, ~\mu,\nu\in\mathbb{R}$ such, that $\alpha < \mu < \beta$ (to be more precise $\mathrm{St}(\alpha, \beta)=(\alpha, \beta)\times\mathrm{i}\,\mathbb{R}\subset\mathbb{C}$, thus a vertical strip parallel with the  imaginary axis intersecting the real axis in $\alpha$ a $\beta$). For {$f\in L_{lok}\left(\mathbb{R}_+\right)$} the Mellin transform is defined as
\begin{equation}\label{MT}
   \mathfrak{M}(f\left(x\right))(s) = f^{*}\left(s\right) = \int\limits_{0}^{\infty} f\left(x\right) x^{s-1} \dd x.
\end{equation}
Largest  open strip $\mathrm{St}(\alpha, \beta)$ of its convergence is called {\it fundamental strip}. Since $|x^{\mathrm{i}\nu}|=1$, we have $|f^*(\mu+\mathrm{i}\nu)|\leq ||f||_{L_{\{\mu\}}(\mathbb{R}_+)}, ~\forall \nu\in\mathbb{R}$.
{To emphasize for fixed $\mu=\Re(s)$ notation $\mathfrak{M}_\mu$ of transform \eqref{MT} is used. It is well known that (absolute) convergence on $\alpha<\mathrm{Re}(s)<\beta$ is determined by asymptotical behaviour near 0 and $\infty$:
$$\alpha=\inf\{\alpha^*: f=\mathcal{O}(x^{-\alpha^*}), ~x\to0^+\},$$
$$\beta=\sup\{\beta^*: f=\mathcal{O}(x^{-\beta^*}), ~x\to\infty\}.$$

A lot of examples can be found in monograph \cite{oberhettinger} consisting of, a.o.,  the Mellin and inverse Mellin transformation tables. We will now present basic properties we need. The set of holomorphic functions on  open $O\subseteq \mathbb{C}^n$ is denoted as $\mathcal{H}(O).$

\begin{thm}[\cite{DAMT}, Pr. 1]\label{skal}
\begin{enumerate}[a)]
    \item Mellin transform is bounded linear operator on $ L_{(\alpha,\beta)}(\mathbb{R}_+)$, so that $|f^*(\mu+\mathrm{i}\,\cdot)|\leq ||f||_{L_{\{\mu\}}(\mathbb{R}_+)}$.
\item If $c>0$ and $f\in L_{(\alpha,\beta)}(\mathbb{R}_+)$, then
$f(cx)\in L_{(\alpha,\beta)}(\mathbb{R}_+)$ a $f(x^c)\in L_{(c\alpha,c\beta)}(\mathbb{R}_+)$, whereas 
$$\mathcal{M}(f(cx))(s)=c^{-s}\mathcal{M}(f(x))(s), ~s\in\mathrm{St}(\alpha,\beta),$$
$$\mathcal{M}(f(x^c))(s)=c^{-1}\mathcal{M}(f(x/c))(s), ~s\in\mathrm{St}(c\alpha,c\beta).$$
\end{enumerate}
\end{thm}

\begin{thm}[\cite{DAMT}, Th. 1.]\label{ThH}
If $f\in L_{(\alpha,\beta)}(\mathbb{R}_+),$ then $f^*\in \mathcal{H}(\mathrm{St}(\alpha,\beta))$.
\end{thm}

\begin{thm}[\cite{yakubovich}, Th. 1.15.]
If $f\in L^p_{\{\mu\}}(\mathbb{R}_+), ~p\in(1,2]$ for some  ~$\mu\in\mathbb{R}$, then
$\mathfrak{M}_\mu(f)$ exists and belongs into $L^{p/(p-1)}(\mathbb{R})$.
\end{thm}

E.g. $f\in L^2_{\{\mu\}}(\mathbb{R}_+)$ yields $\mathfrak{M}_\mu(f)\in L^{2}(\mathbb{R})$. 
On the other hand, when we talk about transformation, we also need to talk about its inversion and we naturally ask if $\mathfrak{M}^{-1}(\mathfrak{M}(f))=f$ ? Here we have to emphasize that the Mellin transform does not map functions from $ L_{(\alpha,\beta)}(\mathbb{R}_+)$
into $\mathcal{H}(\mathrm{St}(\alpha,\beta))$ surjectively. Indeed, there is no image, for example, for a holomorphic function $f\equiv 1$. Necessary condition for $F\in \mathcal{H}(\mathrm{St}(\alpha,\beta))$ to be the Mellin transform of a function from $L_{(\alpha,\beta)}(\mathbb{R}_+)$, is 
\begin{equation}\label{NP}
    \lim_{|\nu|\to\infty}F(\mu+\mathrm{i}\nu)=0
\end{equation}
uniformly with respect to $\mu\in[a,b] ~\forall [a,b]\subset (\alpha,\beta)$. From the further it follows that the condition  \eqref{NP} it is also sufficient if $F(\mu+\mathrm{i}\,\cdot)\in L(\mathbb{R}) ~\forall \mu\in (\alpha,\beta)$. For $F(\mu+\mathrm{i}\,\cdot)\in L(\mathbb{R})$ is the inverse Mellin transform defined as
\begin{equation}\label{MINV}
    \mathfrak{M}^{-1}_\mu(F(s))(x):=\frac{x^{-\mu}}{2\pi}\int_\mathbb{R}F(\mu+\mathrm{i}\,\nu)\,x^{-\mathrm{i}\,\nu}\dd\nu, ~x>0
\end{equation}

It is well defined as the next theorem says.

\begin{thm}[\cite{DAMT}, Pr. 5., Lm. 4.]\label{MTT}
\begin{enumerate}[a)]
    \item  If $F(\mu+\mathrm{i}\,\cdot)\in L(\mathbb{R})$ for some $\mu\in\mathbb{R}$, then
    $$    |\mathfrak{M}^{-1}_\mu(F(s))(x)|\leq \frac{x^{-\mu}}{2\pi}\int_\mathbb{R}|F(\mu+\mathrm{i}\nu)|\dd \nu<\infty, ~x>0$$
    thus, the inverse Mellin transform \eqref{MINV} is well defined.
    \item If $F(\mu+\mathrm{i}\,\cdot)\in L(\mathbb{R}), \forall \mu\in(\alpha, \beta), ~F\in\mathcal{H}(\mathrm{St}(\alpha,\beta))$ and holds \eqref{NP}, then $\mathfrak{M}^{-1}_\mu(F)$ is independent on the choice of $\mu\in(\alpha,\beta)$ and $\mathfrak{M}^{-1}_\mu(F)\in L_{(\alpha,\beta)}(\mathbb{R}_+)$.
\end{enumerate}
\end{thm}

Is is also fact, that if the function $F$ is holomorphic and vanishes sufficiently fast for $\Im(s)\to\pm\infty$, then the Cauchy integral theorem yields that the inverse\footnote{It is good to realize that from a practical point of view the inverse Mellin transformation formula \eqref{Minv} is not easy to use and usually leads to relatively complicated calculations (using, for example, the theory of complex analysis).} \eqref{MINV} is in fact a (complex) path integral (along a line parallel to the $y$-axis) given by
\begin{equation}\label{Minv}
    \mathfrak{M}^{-1}_\mu(F(s))(x) =\frac{1}{2 \pi \mathrm{i}} \int_{\mu-i \infty}^{\mu+i \infty} x^{-s} F(s)\dd s:=\lim_{R_1,R_2\to\infty}\frac{1}{2 \pi \mathrm{i}} \int_{\mu-i R_1}^{\mu+i R_2} x^{-s} F(s)\dd s.
\end{equation} }

If we add boundedness to holomorphy property, we can say something about the continuity of images of the transform.

\begin{thm}[\cite{poularikas}, Th. 12.1.]\label{ohran}
If $F\in \mathcal{H}(\mathrm{St}(\alpha,\beta))$ and if does exist $K>0$ such that $|F(s)|\leq K|s|^{-2}$, then $\mathfrak{M}^{-1}(F)\in C(\mathbb{R}_+).$
\end{thm}

We may also be interested when the image of the transform will be in some of the spaces $L^{q}_{\{\mu\}}(\mathbb{R})$.

\begin{thm}[\cite{yakubovich}, Th. 1.16.]\label{LP}
If $F(\mu+\mathrm{i}\,\cdot)\in L^p(\mathbb{R}), ~p\in(1,2]$ for some $\mu\in\mathbb{R}$, then $\mathfrak{M}^{-1}_\mu(F)$ exists and belongs into $L^{p/(p-1)}_{\{\mu\}}(\mathbb{R}_+)$.
\end{thm}

Here is an important inversion theorem.

\begin{thm}[\cite{DAMT}, Th. 7., 8.]\label{VoINV}
\begin{enumerate}[a)]
    \item If $f\in L_{\{\mu\}}(\mathbb{R}_+)$ and $f^*(\mu+\mathrm{i}\,\cdot)\in L(\mathbb{R})$ for some ~$\mu\in\mathbb{R}$, then $\mathfrak{M}^{-1}_\mu(\mathfrak{M}_\mu f)(x)=f(x)$ a.e. on $\mathbb{R}_+$ (if moreover $f\in C(\mathbb{R}_+)$, then it holds $\forall x\in\mathbb{R}_+$).
    \item If $f\in L_{(\alpha,\beta)}(\mathbb{R}_+)$ and $f^*(\mu+\mathrm{i}\,\cdot)\in L(\mathbb{R})$ for every ~$\mu\in(\alpha,\beta)$, then $\mathfrak{M}^{-1}_{\mu_1}=\mathfrak{M}^{-1}_{\mu_2}, ~\forall ~\mu_1, \mu_2\in(\alpha,\beta).$
    \item If $f, g\in L_{\{\mu\}}(\mathbb{R}_+)$ for some ~$\mu\in\mathbb{R}$ such that $f^*(\mu+\mathrm{i}\,\nu)=g^*(\mu+\mathrm{i}\,\nu), ~\forall \nu\in\mathbb{R}$, then
    $f=g$ a.e. on $\mathbb{R}_+$.
\end{enumerate}
\end{thm}

Remember that it is often possible to use tables from a monograph \cite{oberhettinger}. It is also possible to use approximate methods or numerical methods, e.g. \cite{Tsamasphyros1976}, \cite{GABUTTI1991191}.

\end{document}